\documentclass[11pt]{article}
    
\usepackage{amsmath,amsfonts,amsthm,amscd,amssymb,graphicx}

\numberwithin{equation}{section}

\usepackage{xcolor}

\newtheorem{theorem}{Theorem}[section]

\newtheorem{lemma}[theorem]{Lemma}

\newtheorem{proposition}[theorem]{Proposition}

\newtheorem{remark}[theorem]{Remark}

\def\eps{\varepsilon }

\newcommand{\RR}{\mathbb{R}}

\newcommand{\CC}{\mathbb{C}}

\newcommand{\cA}{\mathcal{A}}

\def\beq{\begin{equation}}
\def\eeq{\end{equation}}
\def\bb1{{1\!\!1}}
\newcommand{\pa}{{\partial}}

\def\cA{\mathcal{A}}

\def\cit{{\Bbb C}}

\def\eps{\varepsilon}

\def\OS{\mathrm{Orr}}
\def\Ray{\mathrm{Ray}}
\def\Airy{\mathrm{Airy}}

\begin{document}

\title{Green function for linearized Navier-Stokes around a boundary shear layer profile for long wavelengths} 

\author{Emmanuel Grenier\footnotemark[1]   \and Toan T. Nguyen\footnotemark[2]
}

\maketitle

\renewcommand{\thefootnote}{\fnsymbol{footnote}}

\footnotetext[1]{Equipe Projet Inria NUMED,
 INRIA Rh\^one Alpes, Unit\'e de Math\'ematiques Pures et Appliqu\'ees., 
 UMR 5669, CNRS et \'Ecole Normale Sup\'erieure de Lyon,
               46, all\'ee d'Italie, 69364 Lyon Cedex 07, France. Email: Emmanuel.Grenier@ens-lyon.fr}

\footnotetext[2]{Department of Mathematics, Penn State University, State College, PA 16803. Email: nguyen@math.psu.edu. TN's research was supported by the NSF under grant DMS-1764119 and an AMS Centennial Fellowship.}



\begin{abstract}

This paper is the continuation of a program, initiated in \cite{GrN1,GrN2}, 
to derive pointwise estimates on the Green function of Orr Sommerfeld equations.
In this paper we focus on long wavelength perturbations, more precisely
horizontal wavenumbers $\alpha$ of order $\nu^{1/4}$, which correspond to the lower boundary of the instability area
for monotonic profiles.

\end{abstract}



\section{Introduction}


We are interested in the study of linearized Navier Stokes around a 
given fixed profile $U_s = (U(z),0)$ in the inviscid limit $\nu \to 0$. 
Namely, we consider the following set of equations
\beq \label{Nlin1}
\partial_t v + U_s \cdot \nabla v + v\cdot  \nabla U_s + \nabla p -  \nu \Delta v = 0,
\eeq
\beq \label{Nlin2}
\nabla\cdot  v = 0,
\eeq
where $0<\nu \ll1$, posed on the half plane $x \in \RR$, $z > 0$,
with the no-slip boundary conditions
\beq \label{Nlin3}
v = 0 \quad \hbox{on} \quad z = 0 .
\eeq

The linear problem \eqref{Nlin1}-\eqref{Nlin3} is a very classical problem 
that has led to a huge physical and mathematical literature, focussing in particular on the linear stability, 
on the dispersion relation, on the study of eigenvalues and eigenmodes, and on the onset of nonlinear instabilities and turbulence \cite{Reid,Schlichting}. We also mention several efforts in proving linear to nonlinear stability and instability around shear flows in the small viscosity limit \cite{GR08,DGV2,Gr1,GGN1,GrN5}.   

Throughout this paper, we will assume that $U(z)$ is holomorphic near $z = 0$, that $U(0) = 0$, that $U'(0) > 0$, that $U(z) > 0$ 
for any $z > 0$, and that $U$ converges
exponentially fast at $\infty$, to some positive constant $U_+$
$$
0 < U_+ = \lim_{z\to \infty} U(z) < \infty,
$$
as well as all its derivatives (which converge to $0$). Note in particular that this class of profiles includes for instance the exponential profile
$$
U(z) = U_+ (1 - e^{- \beta z})
$$
where $\beta > 0$. As such a profile has no inflection point, according to Rayleigh's inflection criterium, 
it is stable with respect to linearized Euler equations.
However, strikingly, a small viscosity has a destabilizing effect. 
That is, {\em all such shear profiles are unstable for large enough Reynolds numbers $\nu^{-1}$} \cite{GGN2,GGN3}. 

More precisely, for such shear flows there exist lower and upper marginal stability branches 
$\alpha_\mathrm{low}(\nu)\sim \nu^{1/4}$ and  $\alpha_\mathrm{up}(\nu)\sim \nu^{1/6}$,  
so that whenever the horizontal wave number $\alpha$ belongs to $[\alpha_\mathrm{low}(\nu),\alpha_\mathrm{up}(\nu)]$, 
the linearized Navier-Stokes equations about this shear profile have an eigenfunction and a corresponding eigenvalue $\lambda_\nu$ 
with 
\begin{equation}\label{growthrate0} 
\Re \lambda_\nu \sim \nu^{1/2} .
\end{equation}
 Heisenberg \cite{Hei,HeiICM}, then Tollmien and C. C. Lin \cite{Lin0,LinBook} 
were among the first physicists to use asymptotic expansions to study this spectral instability. We refer to
Drazin and Reid \cite{Reid} and Schlichting \cite{Schlichting} for a complete account of the physical literature on the subject, 
and to \cite{GGN2,GGN3} for a complete mathematical proof of this instability. 

To study the linear stability of $U_s$ we first introduce the vorticity of the perturbation
$$
\omega = \nabla \times v = \partial_z v_1 - \partial_x v_2,
$$ 
which leads to
\begin{equation}\label{EE-vort}
 (\partial_t + U\partial_x) \omega + v_2 U'' - \nu \Delta \omega = 0
\end{equation}
together with $v = \nabla^\perp \phi$ and $\Delta \phi = \omega$, where $\phi$ is the related stream function. 
The no-slip boundary condition \eqref{Nlin3} becomes $\phi = \partial_z \phi =0$ on $\{z=0\}$. 

We then take the Fourier transform in the tangential variables with Fourier variable $\alpha$ and the Laplace
transform in time with dual variable $-i \alpha c$, following the traditional notations. In other words we study
solutions of linearized Navier Stokes equations which are of the form
$$
v = \nabla^{\perp} \Bigl( e^{i \alpha (x - c t)} \phi_\alpha(z) \Bigr) 
$$
This leads to the 
classical Orr-Sommerfeld equation, 
\beq \label{OS}
\OS_{\alpha,c}(\phi_\alpha) := 
- \eps \Delta_\alpha^2 \phi_\alpha + (U - c) \Delta_\alpha \phi_\alpha - U'' \phi_\alpha = 0
\eeq
where
$$
\eps = {\nu \over i \alpha},
$$
together with the boundary conditions
\beq \label{OS2}
{\phi_\alpha}_{\vert_{z=0}} = \partial_z{ \phi_\alpha}_{\vert_{z=0}} = 0, \qquad \lim_{z\to \infty}\phi_\alpha(z) =0,
\eeq
and where
$$
\Delta_\alpha = \partial_z^2 - \alpha^2 .
$$

The aim of this paper is to give bounds on the Green function of the Orr Sommerfeld equation 
when $\alpha$ is of order  $\nu^{1/4}$ and $c$ is of the same order, which corresponds to one of the boundaries
of the instability area. This restricted study appears to be sufficient to construct linear and nonlinear instabilities 
for the full nonlinear Navier Stokes equations \cite{GrN1,GrN5}. 

To construct the Green function we first construct 
two approximate solutions $\phi_{s,\pm}^{app}$ with a "slow behavior", and two
approximate solutions $\phi_{f,\pm}^{app}$ with a "fast behavior" (the "-" solutions going to $0$ as $z$ goes to $+ \infty$).
These approximate solutions (and in fact exact solutions) have already been constructed in \cite{GGN3}. In this paper
we propose a much simplified and much shorter construction of these approximate solutions, making the current paper
self contained.

The slow approximate solutions will be solutions of the Rayleigh equation
\beq \label{Rayleigh}
(U - c) \Delta_\alpha \phi - U'' \phi = 0
\eeq
with boundary condition $\phi(0) = 0$. They will be constructed by perturbation of the case
$\alpha = 0$ where the Rayleigh equation degenerates in
\beq \label{limitRay}
Ray_0(\phi) = (U - c) \partial_z^2 \phi - U'' \phi .
\eeq
The main observation is that $\phi_{1,0} = U - c$ is a particular of (\ref{limitRay}). 
Let $\phi_{2,0}$ be the other solution of this equation such that the Wronskian $W[\phi_{1,0},\phi_{2,0}]$ equals $1$. 
We will construct approximate solutions to the Orr Sommerfeld equation which satisfy
\beq \label{defiEvans1}
\phi_{s,-}^{app}(0) = U(0) - c + \alpha U_+^2 \phi_{2,0}(0) + O(\alpha^2), 
\eeq
\beq \label{defiEvans2}
\partial_z \phi_{s,-}^{app}(0) = U'(0) + O(\alpha).
\eeq
The "fast approximate solutions" will emerge in the balance between $- \eps \Delta_\alpha^2 \phi$
and $(U - c) \Delta_\alpha \phi$. Keeping in mind that $\alpha$ is small, they will be constructed starting
from solutions of the simplified equation
$$
- \eps \partial_z^4 \phi + (U - c) \partial_z^2 \phi = 0.
$$
As $c$ is small, and as $U'(0) \ne 0$, there exists a unique $z_c \in \cit$ near $0$ such that
\beq \label{critic1}
U(z_c) = c .
\eeq
Such a $z_c$ is called a "critical layer" in the physics literature.
It turns out that all the instability is driven by what happens near this critical layer.
Near $z_c$, equation (\ref{critic1}) is a perturbation of the Airy equation
\beq \label{Airy}
- \eps \partial_z^2 \psi + U'(0) (z - z_c) \psi = 0
\eeq
posed on $\psi = \partial_z^2 \phi$. The fast approximate solutions are thus constructed as perturbations
of second primitives of classical Airy functions. This construction will be detailed in Section \ref{sectionAiry}, where we will
construct two approximate solutions $\phi_{f,\pm}^{app}$ to Orr Sommerfeld equation, with a fast behavior and with
\beq \label{valuephi}
\phi_{f,-}^{app}(0) = Ai(2,- \gamma z_c)) + O(\nu^{1/4})
\eeq
\beq \label{valuephid}
\partial_z \phi_{f,-}^{app}(0) = \gamma Ai(1,- \gamma z_c) + O(1),
\eeq
where
\beq \label{valuegamma}
\gamma = \Bigl(  {i \alpha U'(z_c) \over \nu} \Bigr)^{1/3} = O(\nu^{-1/4}),
\eeq
and where $Ai(1,.)$ and $Ai(2,.)$ are the first and the second primitives of the classical Airy function $Ai$.
We now introduce the Tietjens function, defined by
$$
Ti(z) = {Ai(1,z) \over Ai(2,z)} .
$$
Tietjens function is a classical special function in physics, precisely known and tabulated.
Then
\beq \label{Tiet}
{\partial_z \phi_{f,-}^{app}(0) \over \phi_{f,-}^{app}(0)} = \gamma Ti(- \gamma z_c) + O(1) .
\eeq
In this paper we will bound the Green function of Orr Sommerfeld equations.
More precisely, for each fixed $\alpha \in \RR_+$ and $c\in \CC$, 
we let $G_{\alpha,c}(x,z)$ be the corresponding Green kernel of the Orr Sommerfeld problem. 
By definition, for each $x\in \RR$ and $c\in \CC$, $G_{\alpha,c}(x,z)$ solves 
$$ 
\OS_{\alpha,c}(G_{\alpha,c} (x,\cdot)) = \delta_x (\cdot)
$$
on $z\ge 0$, together with the boundary conditions:
$$
G_{\alpha,c}(x,0) = \partial_z G_{\alpha,c} (x,0) =0, \qquad \lim_{z\to \infty} G_{\alpha,c}(x,z) =0.
$$
That is, for $z\not = x$, the Green function $G_{\alpha,c} (x,z)$ solves the homogenous Orr-Sommerfeld equations, 
together with the following jump conditions across $z=x$:
$$ 
[\partial_z^kG_{\alpha,c}(x,z)]_{\vert_{z=x}} = 0, \qquad [\epsilon\partial_z^3 G_{\alpha,c}(x,z)]_{\vert_{z=x}} = -1
$$ 
for $k=0,1,2$. 
Here, the jump $[f(z)]_{\vert_{z=x}}$ across $z=x$ is defined to be 
the value of the right limit subtracted by that of the left limit as $z\to x$.  

Let
$$
\mu_f(z) = \sqrt{ { U(z) - c \over \eps}},
$$
taking the square root with a positive real part.
Note that
\beq \label{muf}
| \mu_f(z) | \ge \Bigl| \sqrt{ { \Im c \over \eps} } \Bigr| 
= \Bigl| \sqrt{ \alpha \, \Im c \over \nu} \Bigr| = O(\nu^{-1/4}).
\eeq
The  main result in this paper is as follows. 

\begin{theorem}\label{theo-GreenOS-stable} 
Let $\sigma_0$ be arbitrarily small.
Let $\alpha = O(\nu^{1/4})$, and let $c = O(\nu^{1/4})$, with $| \Im c | \ge \sigma_0 \nu^{1/4}$, 
such that
\beq \label{disp}
| W[ \phi_{s,-}^{app}, \phi_{f,-}^{app} ]| \ge \sigma_0.
\eeq
Let $G_{\alpha,c}(x,z)$ be  the Green function of the Orr-Sommerfeld problem. 
Then, there exists a smooth function $P(x)$ and there are universal positive constants $\theta_0, C_0$ so that 
\begin{equation}\label{est-GrOS-stable}
\begin{aligned}
 |G_{\alpha,c}(x,z) -  {P(x) \phi_{s,-}(z) \over \nu^{1/4}} |  &\le 
  {C_0 \over \eps \mu_f^2(x)} \Big( e^{-\theta_0  | \alpha | |x-z|} + 
 {1 \over | \mu_f(x) |} e^{ - \theta_0 | \int_x^z \Re \mu_f(y) dy | } \Big)
\end{aligned}
\end{equation}
uniformly for all $x,z\ge 0$. 
Similarly, 
\begin{equation}\label{est-GrOS-stable2}
\begin{aligned}
 | \partial_z G_{\alpha,c}(x,z)  - { P(x) \partial_z \phi_{s,-}(z) \over \nu^{1/4} }  |  &\le 
  {C_0 \over \eps \mu_f^2(x)} \Big( e^{-\theta_0  |\alpha| |x-z|} + 
 {| \mu_f(z) | \over | \mu_f(x) |} e^{ - \theta_0  | \int_x^z \Re \mu_f(y) dy | } \Big)
\end{aligned}
\end{equation}
 \end{theorem}
 
Let us comment (\ref{disp}). We have
$$
W[ \phi_{s,-}^{app}, \phi_{f,-}^{app} ]
= \gamma \psi_{s,0}^{app}(0) Ti(- \gamma z_c)  \phi_{f,-}^{app}(0)   - \partial_z \phi_{s,-}^{app}(0)  \phi_{f,-}^{app}(0) 
$$
$$
= - \Bigl( \gamma c Ti(- \gamma z_c) + U'(0) \Bigr) Ai(2,-\gamma z_c) + O(\nu^{1/4})
$$
Note that both terms under the brackets are of order $O(1)$, since $\gamma c$ is of order $O(1)$.
The Wronskian vanishes if there exists a linear combination of
$\phi_{s,-}^{app}$ and $\phi_{f,-}^{app}$ which satisfies the boundary conditions, namely if there exists an approximate
eigenmode of $\OS_{\alpha,c}$ (recalling that $\phi_{s,-}^{app}$ and $\phi_{f,-}^{app}$ are only approximate solutions of $\OS_{\alpha,c}$).
We have to remain away from such approximate modes, since nearby there exists true eigenmodes where $\OS_{\alpha,c}$
is no longer invertible. Note that $\sigma_1$ may be taken arbitrarily small.

Note that in this Theorem we are at a distance $O(\nu^{1/4})$ from a simple eigenmode $\psi_0$. It is therefore expected that
$\OS_{\alpha,c}$ is of order $O(\nu^{-1/4})$ and that
\beq \label{inverseOrr}
\OS^{-1}_{\alpha,c}(\psi) = \nu^{-1/4} \Bigl( \int_0^{ +\infty} P(z) \psi(z) dz \Bigr) \psi_0 + O(1) .
\eeq
As $\psi_0 = \phi_{s,-} + O(\nu^{1/4})$, $G_{\alpha,c}$ is only bounded
by $O(\nu^{-1/4})$, and its main component is $ \nu^{-1/4} P \phi_{s,-}$.


\section{The Airy operator \label{sectionAiry}}


In this section, we construct two approximate solutions of Orr Sommerfeld equation, called
$\phi_{f,\pm} = \phi_{f,\pm}^{app}$, with fast increasing or decreasing behaviors. For these approximate solutions, it turns out
that the zeroth order term $U'' \phi_{f,\pm}$ may be neglected. 
Moreover, as $\alpha$ is small, $\alpha^2$ terms may also be neglected.
This simplifies the Orr Sommerfeld operator in the so called modified Airy operator defined by
\begin{equation}\label{Airy-de}
\Airy = \cA \partial_z^2 ,
\end{equation}
where
\begin{equation}\label{def-cA}
\mathcal{A}: = - \eps \partial_z^2 + (U-c) .
\end{equation}
Note that
\beq \label{Airy-de-2}
\OS_{\alpha,c} = \Airy + \hbox{OrrAiry} 
\eeq
where
$$
\hbox{OrrAiry} = 2 \eps \alpha^2 \partial_z^2 - \eps \alpha^4  - \alpha^2 (U-c) - U'' .
$$
Note also that $U-c$ behaves like $U'(z_c) ( z - z_c)$ for $z$ near $z_c$, hence $\cA$ is very similar to the classical Airy
operator $\partial_z^2 - z$ when $z$ is close to $z_c$. The main difficulty lies in the fact that the "phase" $U(z) - c$ almost
vanishes when $z$ is close to $\Re z_c$, hence we have to distinguish between two cases: $z \le \sigma_1$ and
$z \ge \sigma_1$ for some small $\sigma_1$. 
The first case is handled through a Langer transformation, which reduces (\ref{Airy-de}) to the classical
Airy equation. The second case may be treated using a classical WKB expansion.

We will prove the following proposition. 

\begin{proposition} \label{prop-maphif}
There exist two smooth functions $\phi^{app}_{\pm}(z)$ such that
\beq \label{fastapp1}
|{\cal A} \partial_z^2 \phi^{app}_{\pm} | \le C \nu^N | \phi^{app}_{\pm} |,
\eeq
\beq  \label{fastapp1b}
| \OS_{\alpha,c}(\phi_{\pm}^{app}) | \le C | \phi^{app}_{\pm} |,
\eeq
for arbitrarily large $N$.
Moreover for $z \gg \nu^{1/4}$ and for $k =1,2,3$, as $\nu \to 0$,
\beq \label{fastapp1c}
{\partial_z^k \phi^{app}_-(z) \over \phi^{app}_-}(z) \sim (-1)^k \mu_f^k(z),
\eeq
and similarly for $\phi^{app}_+$ without the $(-1)^k$ factor.
For $k = 1, 2, 3$ and  any $x_1 < x_2$,
\beq \label{fastapp2}
\Bigl| {\phi^{app}_{+}(x_2) \over \phi^{app}_{+}(x_1)} \Bigr| \le C \exp \Bigl(
\int_{x_1}^{x_2} \Re \mu_f(y) dy \Bigr) 
\eeq
and similarly for $\phi^{app}_-$.
\end{proposition}

To prove this proposition we construct $\psi^{app}_\pm = \partial_z^2 \phi^{app}_\pm$ for 
$z < z_c$ in Section \ref{Airypoints} using the Langer's transformation
introduced in (\ref{sec-Langer}) and for $z > z_c$ in Section \ref{sec-WKB} using the classical WKB method. We then match these
two constructions in Section \ref{sec-match}, integrate them twice in Section \ref{AA}
and detail the Green function of Airy operator in Section \ref{GreenAiry}.


\subsection{A primer on Langer's transformation}\label{sec-Langer}


The first step is to construct approximate solutions to ${\cal A} \psi = 0$, 
starting from solutions of the genuine Airy equation
$\eps \psi'' = y \psi$, thanks to the so called Langer's transformation that we will now detail.
Let $B(x)$ and $C(x)$ be two smooth functions. 
In $1931$,  Langer introduced the following method to build approximate solutions to the varying coefficient
Airy type equation
\beq \label{nonlinC}
- \eps \phi'' + C(x) \phi = 0
\eeq
starting from solutions to the similar Airy type equation
\beq \label{nonlinB}
- \eps \psi'' + B(x) \psi = 0.
\eeq
We assume that  both $B$ and $C$ vanish at some point $x_0$, and that their derivatives at $x_0$ does not vanish.
Let $\psi$ be any solution to (\ref{nonlinB}). Let $f$ and $g$ be two
smooth functions, to be chosen later. Then 
$$
\phi(x) = f(x) \psi(g(x))
$$ 
satisfies
$$
- \eps \phi'' + C(x) \phi = - \eps f'' \psi -  2 \eps f' \psi' g' - B(g(x))  (g')^2 f \psi - \eps f \psi' g'' +  C(x) f\psi.
$$
Note that $f$ may be seen as a modulation of amplitude and $g$ as a change of phase.
If we choose $g$ such that 
\beq \label{eqg}
B(g(x)) (g')^2 = C(x) 
\eeq
and $f$ such that
\beq \label{eqf}
2 f' g' + f g'' = 0,
\eeq
we have
$$
- \eps \phi'' + C(x) \phi = - \eps f'' \psi.
$$
Hence $\phi$ may be considered as an approximate solution to $- \eps \phi'' + C(x) \phi = 0$.

Note that (\ref{eqf}) may be solved, yielding
\beq \label{eqf1}
f(x) = {1 \over \sqrt{g'(x)}} .
\eeq
Now let $B_1$ be the primitive of $\sqrt{B}$ which vanishes at $x_0$
and let $C_1$ be the primitive of $\sqrt{C}$ which vanishes at $x_0$.
Then (\ref{eqg}) may be rewritten as
\beq \label{eqg2}
B_1(g(x))  = C_1(x).
\eeq
Note that both $B_1$ and $C_1$ behave like $C_0(x - x_0)^{3/2}$ near $x_0$. Hence (\ref{eqg2}) may be solved for
$x$ near $x_0$. This defines a smooth function $g$ which satisfies $g(x_0) = x_0$. Moreover if $B'(x_0) = C'(x_0)$
then $g'(x_0) = 1$.



\subsection{Airy critical points \label{Airypoints}}


In this section we  use Langer's transformation to construct approximate solutions to ${\cal A} \psi = 0$ starting
from solutions of the genuine Airy equation.

Let $c$ be of order $\nu^{1/4}$. Then there exists an unique $z_c \in \cit$ near $0$ such that $U(z_c) = c$. 
Note that $z_c$ is also of order
$\nu^{1/4}$ since $U'(0) \ne 0$. Expanding $U$ near $z_c$ at first order we get the approximate equation
\beq \label{order2l}
- \eps \partial_z^2 \psi + U'(z_c) (z - z_c) \psi = 0 
\eeq
which is the classical Airy equation. Let us assume that $\Re U'(z_c) > 0$, the opposite case being similar.
A first solution is given by  
$$
A(z) = Ai ( \gamma (z - z_c) )
$$
where $Ai$ is the classical Airy function, solution of $Ai'' = x Ai$, and where
$\eps \gamma^3 =  U'(z_c)$, namely
$$
\gamma = \Bigl(  {i \alpha U'(z_c) \over \nu} \Bigr)^{1/3} .
$$
Note that since $\alpha$ is of order $\nu^{1/4}$, $\gamma$ is of  order $\nu^{-1/4}$ and that 
$$
\arg(\gamma) =  + \pi / 6 + O(\nu^{-1/4}).
$$
Moreover, as $x$ goes to $\pm \infty$, with argument $i \pi / 6$,
$$
Ai(x) \sim {1 \over 2 \sqrt{\pi}}  { e^{- 2 x^{3/2}  / 3} \over x^{1/4}} .
$$
In particular, $Ai'(x) / Ai(x) \sim - x^{1/2}$ for large $x$. Hence, as $\gamma(z - z_c)$ goes to infinity, $A(z)$ goes to $0$ and 
\beq \label{asymplog}
{A'(z) \over A(z)} \sim - \gamma^{3/2} (z - z_c)^{1/2} = 
-  \Bigl(  {i \alpha U'(z_c) \over \nu} \Bigr)^{1/2} (z - z_c)^{1/2} \sim - \sqrt{B(z)},
\eeq
with 
$$
B(z) = \eps^{-1} U'(z_c) (z - z_c).
$$ 
More precisely, we get 
$$
{A'(z) \over A(z)}  = - \sqrt{B(z)} (1 + O(\nu^{1/4}))
$$
for $z \gg \nu^{1/4}$ (on which $\gamma (z-z_c) \gg1$). 

An independent solution is given by $Ci (\gamma (z - z_c))$ where
$$
Ci = - i \pi (Ai + i Bi),
$$
with $Bi(\cdot)$ being the other classical Airy function. In this case $| Ci (\gamma (z - z_c)) |$ goes to $+ \infty$ as $z - z_c$ goes to $+ \infty$,
with a plus instead of the minus in the corresponding formula (\ref{asymplog}).

We now use Langer's transformation introduced in the previous section. As $U(z)$ and $U'(z_c) (z - z_c)$ vanish at the same point
with the same derivative at that point, we use Langer's transformation with
$$
C(z) = \eps^{-1} (U(z) - c)
$$
and
$$
B(z) = \eps^{-1} U'(z_c) (z - z_c).
$$
Then, $g(z)$ is locally well defined, for $0 \le z \le \sigma_1$ for some positive $\sigma_1$. 
Moreover $g( z_c) = z_c$ and $g'( z_c) = 1$. Now
$$
\tilde Ai (z) = {1 \over \sqrt{g'(z)}} Ai \Bigl(\gamma g(z) \Bigr) 
$$
and
$$
\tilde Ci(z) = {1 \over \sqrt{g'(z)}} Ci \Bigl(\gamma g(z) \Bigr) 
$$
are two approximate solutions of $\cA \phi = 0$ in the sense that
$$
\cA \tilde Ai = - \eps f'' Ai(\gamma g(z)) , \qquad \cA \tilde Ci = - \eps f'' Ci(\gamma g(z)).
$$
Note that the error term is of order $\eps \sim \nu^{3/4}$. 
Note also that at first order, for $z$ of order $\nu^{1/4}$,
$\tilde Ai(z)$ equals $Ai(\gamma (z - z_c))$ since $g'(z_c) = 1$. 

Moreover, for $z \gg \nu^{1/4}$, using (\ref{eqg}), we get
\beq \label{polarloin}
{\partial_z \tilde Ai(z) \over \tilde Ai(z)} \sim  g'(z) {A'( g(z)) \over A( g(z))} \sim - g'(z) \sqrt{B(g(z))} 
\sim  - \sqrt{C(z)} \sim - \mu_f(z),
\eeq
and more precisely
$$
{\partial_z \tilde Ai(z) \over \tilde Ai(z)} \sim - \mu_f(z) (1 + O(\nu^{1/4})) .
$$
Similarly for $z \gg \nu^{1/4}$, we get
\beq \label{polarloin2}
{\partial_z^k \tilde Ai(z) \over \tilde Ai(z)} \sim (-1)^k \mu_f^k(z) .
\eeq


\subsection{Away form the critical layer \label{sec-WKB}}


If $z - z_c$ is small then $g$ is well defined, precisely on $[0,\sigma_1]$ for some small $\sigma_1$ as in the previous section. However, if $z > \sigma_1$, then Langer's transformation is no longer useful, and we may directly use a WKB expansion.
We look for solutions $\psi$ of the form
$$
\psi(z) = e^{ \theta(z) / \eps^{1/2}} 
$$
to the equation $\cA\psi = \eps \partial_z^2 \psi - (U - c) \psi = 0$.
Note that
$$
\eps \partial_z^2 \psi= \Bigl( \theta'^2 + \eps^{1/2} \theta'' \Bigr) \psi.
$$
Hence we look for $\theta$ such that
\beq \label{theta2}
\theta'^2 + \eps^{1/2} \theta'' = (U - c) .
\eeq
As we are only interested in approximate solutions, we solve (\ref{theta2}) in an approximate way, and look
for $\theta$ of the form
$$
\theta = \sum_{i = 0}^M \eps^{i/2} \theta_i 
$$
for some arbitrarily large $M$. The $\theta_i$ may be constructed by iteration, starting from
$$
\theta_0' = \pm \sqrt{U (z) - c} .
$$
If we keep the positive real part to the square root, the $-$ choice leads to a solution going to $0$ at $+ \infty$
and the $+$ choice to a solution going to $+ \infty$ at $+ \infty$.
This construction gives a solution $\psi^{app}_{f,\pm}$ such that
$$
| \cA\psi^{app}_{f,\pm} | \le \nu^N | \psi^{app}_{f,\pm} | ,
$$
where $N$ can be chosen arbitrarily large provided $M$ is sufficiently large.
Note that
\beq \label{pola2}
\partial_z \psi^{app}_{f,\pm}(\sigma_1) = \pm  \mu_f(\sigma_1) (1 + O(\nu^{1/4}) )
\psi^{app}_{f,\pm}(\sigma_1) .
\eeq
More generally,
\beq \label{pola2b}
\partial_z^k \psi^{app}_{f,-}(z) = (-1)^k \mu_f^k(z) \psi^{app}_{f,\pm}(z) (1 + O(\nu^{1/4}))
\eeq
for any $z \ge \sigma_1$ and any $k$, and similarly for $\psi^{app}_{f,+}$.


\subsection{Matching at $z = z_c$ \label{sec-match}}


It remains to match at $z = z_c$ the solutions constructed with the WKB method for $z \ge \sigma_1$ 
with the solutions construct thanks to Langer's transformation for $z \le \sigma_1$. 
We look for constants $a$ and $b$ such that
$$
a {\tilde Ai(z) \over \tilde Ai(\sigma_1)} + b {\tilde Ci(z) \over \tilde Ci(\sigma_1)}
$$
and $\psi^{app}_{f,-} /  \psi^{app}_{f,-}(\sigma_1)$ and their first derivatives match at $z = \sigma_1$, which leads to
$$
a + b  = 1
$$
$$
a {\partial_z \tilde Ai(\sigma_1) \over \tilde Ai(\sigma_1)} + b {\partial_z \tilde Ci (\sigma_1)\over \tilde Ci(\sigma_1)}
= {\partial_z \psi^{app}_{f,-} (\sigma_1) \over  \psi^{app}_{f,-}(\sigma_1)}
$$
We now use (\ref{asymplog}) and (\ref{pola2}) to get $a \sim 1$ and $b = O(\mu_f(\sigma_1)^{-1})$.
We then multiply $a$ and $b$ by $\psi^{app}_{f,-}(\sigma_1)$ to get an extension of $\psi^{app}_{f,-}$
from $z > \sigma_1$ to the whole line.
The construction is similar to extend $\psi^{app}_{f,+}$.


\subsection{From ${\cal A}$ to Airy \label{AA}}


We have now constructed global approximate solutions, that we again call $\psi_{f,\pm}^{app}$.
It remains to solve
\beq \label{doubleprimitive}
\partial_z^2   \phi_{f,\pm}^{app}(z) = \psi_{f,\pm}^{app}(z).
\eeq
Let us focus on the $-$ case, the other being similar.
For $z \ge \sigma_1$, we look for solutions $\phi_{f,\pm}^{app}$ of the form
$$
\phi_{f,\pm}^{app} = h(x)\psi_{f,\pm}^{app} =  h(x) e^{\theta(x) / \eps^{1/2}}
$$
which leads to
$$
h'' + 2 h' \theta'(x) \eps^{-1/2} + h \theta''(x) \eps^{-1/2} + h \theta'^2(x) \eps^{-1} = 1.
$$ 
Hence $h$ may be expanded as a series in $\eps^{1/2}$; namely, 
$$ h(x) = \sum_{i=0}^M \epsilon^{i/2} h_i(x) $$
for some arbitrarily large $M$. The leading term $h_0(x)$ is defined by $$
h_0(x) = {\eps \over \theta'^2(x) },
$$
while the other terms are computed similarly. We may thus write a complete WKB expansion for $\phi_{f,\pm}^{app}$.
In particular
$$
{\phi_{f,-}^{app}(y) \over \phi_{f,-}^{app}(x) } \le e^{- \int_x^y \Re \mu_f(z) dz} 
$$
provided $y > x \ge \sigma_1$. 

For $z < \sigma_1$, we integrate once (\ref{doubleprimitive}) which gives
$$
\partial_z\phi_{f,-}^{app}(z) = \partial_z\phi_{f,-}^{app}(\sigma_1) - \int_z^{\sigma_1} \psi_{f,-}^{app}(t) dt .
$$ 
Now $\psi_{f,-}^{app}$ is a combination of $\tilde Ai$ and $\tilde Ci$ for $z < \sigma_1$.
Let us focus on the $\tilde Ai$ term. We have to study
$$
\int_z^{\sigma_1} \tilde Ai(t) dt =  \int_z^{\sigma_1}
{1 \over \sqrt{g'(t)}} Ai(\gamma g(t)) dt.
$$
Let $s = \gamma g(t)$. Then $ds = \gamma g'(t) dt$, hence
$$
 \int_z^{\sigma_1}
{1 \over \sqrt{g'(t)}} Ai(\gamma g(t)) dt = \gamma^{-1}  \int_{\gamma g(z)}^{\gamma g(\sigma_1)} {1 \over g'(t)^{3/2}} Ai(s) ds.
$$
As $\gamma$ is large, the integral term is equivalent to
$$
{\gamma^{-1} \over g'(z)^{3/2}}   \int_{\gamma g(z)}^{\gamma g(\sigma_1)} Ai(s) ds
\sim {\gamma^{-1}\over g'(z)^{3/2}} \Bigl[ Ai(1,\gamma g(\sigma_1)) - Ai(1,\gamma g(z)) \Bigr]
$$
where we introduced the primitive $Ai(1,x)$ of $Ai$. This leads to
\beq \label{dzphi10}
\partial_z \phi_{f,-}^{app}(z) \sim {\gamma^{-1}\over g'(z)^{3/2}}  Ai(1,\gamma g(z)) .
\eeq
We integrate one again $\partial_z \phi_{f,-}^{app}$ and introduce $Ai(2,x)$, the second primitive of
$Ai$ and obtain
\beq \label{phi10}
\phi_{f,-}^{app}(z) \sim {\gamma^{-2}\over g'(z)^{5/2}}  Ai(2,\gamma g(z)) .
\eeq
The study of $\phi_{f,+}$ is similar.
As the asymptotic expansion of $Ai(x)$ is known, we can compute the asymptotic expansions of $Ai(1,x)$ and
$Ai(2,x)$.


\subsection{End of proof of Proposition \ref{prop-maphif}}


We now multiply $\phi_{f,\pm}^{app}$ by $\gamma^2$, such that, after this normalization, we have
(\ref{valuephi}) and (\ref{valuephid}).
In particular
\beq \label{normalfast}
\phi_{f,\pm}^{app}(0) = O(1) .
\eeq
Using (\ref{asymplog}) and (\ref{pola2b}) we get that
$$
{\partial_z \phi_{f,+}^{app}(z) \over \phi_{f,+}^{app}(z)} = \mu_f(z) (1 + O(\nu^{1/4}))
$$
as soon as $z \gg \nu^{1/4}$. As $\mu_f(z)$ is of order $O(\nu^{-1/4})$ for $z$ of order $\nu^{1/4}$, 
we obtain for any $0 \le z \le z'$,
\beq \label{boundsphif}
\Bigl|{\phi_{f,+}^{app} (z') \over \phi_{f,+}^{app}(z)} \Bigr| \le C \exp \Bigl| \int_z^{z'} \Re \mu_f(s) ds \Bigr| 
\eeq
for some constant $C$, and similarly for $\phi_{f,-}$, which gives (\ref{fastapp2}).

Moreover (\ref{fastapp1}) and (\ref{fastapp1c}) have already been proven.
As $\partial_z \phi_{f,+}^{app}(z)$ is bounded by $C \nu^{-1/4} \phi_{f,+}^{app}(z)$, (\ref{fastapp1}) combined with
(\ref{Airy-de-2}) gives (\ref{fastapp1b}), which ends the proof of Proposition \ref{prop-maphif}.


\subsection{Green function for Airy \label{GreenAiry}}


We will now construct an approximate Green function for the $\Airy$ operator.
We first construct an approximate Green function for ${\cal A}$. Let 
$$
G^{Ai}(x,y) = {1 \over \eps W^{Ai}(x)}  \left\{ \begin{aligned} 
{\psi^{app}_+(y) \over \psi^{app}_+(x)} \quad \hbox{if} \quad y < x,
\\
 {\psi^{app}_-(y) \over \psi^{app}_-(x)} \quad \hbox{if} \quad y > x,
 \end{aligned}\right.
$$
where $W^{Ai}$ is the Wronskian of $\psi^{app}_\pm(x)$. Note that this Wronskian is independent of $x$ and of order
$$
W^{Ai}(x) \sim \gamma = O(\nu^{-1/4}).
$$ 
In particular, we have
$$
G^{Ai}(x,y) = O(\nu^{-1/2})  \exp \Bigl( - C  \Bigl| \int_x^y \Re \mu_f(z) dz \Bigr| \Bigr),
$$
therefore $G^{Ai}$ is rapidly decreasing in $y$ on both sides of $x$, within scales of order $\nu^{1/4}$.
By construction, 
$$
{\cal A} G^{Ai}(x,y) = \delta_x + O(\nu^{3/4}) G^{Ai}(x,y) .
$$
We then integrate twice $G^{Ai}$ in $y$ to get an approximate Green function for the $\Airy$ operator.
More precisely, let
$$
G^{Ai,1}(x,y) = \int_{y}^{+\infty} G^{Ai}(x,z) dz
$$
and similarly for $G^{Airy} = G^{Ai,2}$, the primitive of $G^{Ai,1}$, so that $\pa_y^2 G^{Ai,2}(x,y) = G^{Ai}(x,y)$. We have
$$
G^{Ai,1}(x,y) = O(\nu^{-1/4})  \exp \Bigl( - C  \Bigl| \int_x^y \Re \mu_f(z) dz \Bigr| \Bigr) + O(\nu^{-1/4}) 1_{y < x}
$$
and similarly for $G^{Ai,2}$
 $$
G^{Ai,2}(x,y) = O(1)  \exp \Bigl( -  C \Bigl| \int_x^y \Re \mu_f(z) dz \Bigr| \Bigr) + O(\nu^{-1/4}) 1_{y < x}  x.
$$
Note that, taking into account the fast decay of $G^{Ai}$ near $x$,
\beq \label{GreenAiry2}
\begin{aligned}
\Airy (G^{Ai,2}) &= \delta_x + O(\nu^{3/4}) G^{Ai}(x,y) 
\\&= \delta_x + O(\nu^{1/4} ) \exp \Bigl( -  C \Bigl| \int_x^y \Re \mu_f(z) dz \Bigr| \Bigr) 
\\&= \delta_x + O(\nu^{1/4}).
\end{aligned}\eeq
We define the $AirySolve$ operator by
\beq \label{AirySolve}
AirySolve(f) (y) = \int_0^{+ \infty} G^{Ai,2}(x,y) f(x) dx
\eeq
and the associated error term
\beq \label{ErrorAiry}
\begin{aligned}
ErrorAiry(f) (y) 
& = \int_0^{+ \infty}  O(\nu^{3/4}) G^{Ai}(x,y)  f(x) dx
\end{aligned}
\eeq
the $\Airy$ operator acting on the $y$ variable.
These operators will be used in Section \ref{sec35}.



\section{Rayleigh solutions near critical layers}\label{sec-Rayleigh}


In this section, we construct two approximate solutions $\phi_{s,\pm}^{app}$ to the Orr Sommerfeld equation,
whose modules  respectively go to $+ \infty$ and $0$ as $z \to + \infty$. More precisely, we prove the following Lemma

\begin{lemma}\label{lem-exactphija} For $\nu$ small enough
there exist two independent functions $\phi_{s,\pm}^{app}$ such that
$$
W[\phi_{s,+}^{app},\phi_{s,-}^{app}](z) = 1 +  o(1),
$$
$$
\OS_{\alpha,c}(\phi_{s,\pm}^{app}) = O(\nu^{1/2}).
$$
Furthermore, we have the following expansions in $L^\infty$
$$
\begin{aligned}
\phi_{s,-}^{app} (z)&=  e^{-\alpha z} \Big (U-c + O(\nu^{1/4} )\Big).
\\
\phi_{s,+}^{app} (z)&=  \alpha^{-1} e^{\alpha z} O(1),
\end{aligned}$$
as $z\to \infty$. At $z = 0$, there hold
$$ 
\begin{aligned}
\phi_{s,-}^{app}(0) &=  - c + \alpha {U_+^2 \over U'(0)}   + O(\nu^{1/2})
\\
\phi_{s,+}^{app}(0) &=   - {1 \over U'(0)} +  O(\nu^{1/2}).
\end{aligned}$$
with 
 \end{lemma}
The construction of approximate solutions for Orr Sommerfeld equation starts with the construction of approximate solutions
for the Rayleigh operator.

For small $\alpha$, the construction of solutions to the Rayleigh equation is
a perturbation of the construction for $\alpha = 0$, which is explicit. We will now
detail the construction of an inverse of $Ray_0$ and then of an approximate inverse of $Ray_\alpha$ for small $\alpha$


\subsection{Function spaces}\label{sec-space}


In the next sections we will denote
$$
X^\eta = L_\eta^{\infty} = \Bigl\{ f \quad | \quad \sup_{z \ge 0} | f(z) | e^{\eta z} < + \infty \Bigr\} .
$$
The highest derivative of the Rayleigh equation vanishes at $z=z_c$, since $U(z_c) = c$. 
To handle functions which have large derivatives when $z$ is close to $\Re z_c$, we introduce the space
$Y^\eta$ defined as follows. Note that in our analysis, $z_c$ is never real, so $z - z_c$ never vanishes. 
We are close to a singularity but never reach it.

We say that a function $f$ lies in $Y^\eta$ if for any $z \ge 1$,
$$
|f(z)| + |\partial_z f(z) | + | \partial_z^2 f (z) |  \le C e^{-\eta z}
$$
and if for $z \le 1$,
$$
| \partial_z f(z) | \le C (1 + | \log (z - z_c) |  ) , 
$$
and
$$| \partial_z^2 f(z) | \le C (1 + | z - z_c |^{-1} ).
$$
The best constant $C$ in the previous bounds defines the norm $\| f \|_{Y^{\eta}}$.


\subsection{Rayleigh equation when $\alpha = 0$}


In this section, we  study the Rayleigh operator $\Ray_0$. More precisely, we solve  
\begin{equation}\label{Ray0} 
\Ray_0 (\phi) = (U-c) \partial_z^2 \phi - U'' \phi = f.
\end{equation}
The main observation is that
$$
\Ray_0(U - c) = 0.
$$
Therefore 
$$
\phi_{1,0} = U-c
$$ 
is a first explicit solution. The second one is obtained through the Wronskian equation
$$
W[\phi_{1,0},\phi_{2,0}] = 1 .
$$ 
This leads to  the following Lemma whose proof is given in \cite[Lemma 3.2]{GGN3}
\begin{lemma}[\cite{GGN2,GGN3}] \label{lem-defphi012} 
Assume that $\Im c \not =0$. There exist two independent solutions $\phi_{1,0} = U - c$ and $\phi_{2,0}$ of $\Ray_0(\phi) =0$ 
with unit Wronskian determinant 
$$ 
W(\phi_{1,0}, \phi_{2,0}) := \partial_z \phi_{2,0} \phi_{1,0} - \phi_{2,0} \partial_z \phi_{1,0} = 1.
$$
Furthermore, there exist smooth functions $P(z)$ and $Q(z)$  
with $P(z_c) \ne 0$ and $Q(z_c)\not=0$, so that, near $z = z_c$,
\begin{equation}\label{asy-phi012} 
\phi_{2,0}(z) = P(z) + Q(z) (z-z_c) \log (z-z_c) .
\end{equation}
Moreover
$$
\phi_{2,0}(0) = - {1 \over U'(0)} 
$$
and
\begin{equation}\label{decay-phi012} 
 \partial_z \phi_{2,0}(z)  + \frac{1}{U_+}   \in  Y^{\eta_1}
\end{equation}
for some $\eta_1 > 0$.
 \end{lemma}
Let $\phi_{1,0},\phi_{2,0}$ be constructed as in Lemma \ref{lem-defphi012}. 
Then the Green function $G_{R,0}(x,z)$ 
of the $\Ray_0$ operator can be explicitly defined by 
$$
G_{R,0}(x,z) = \left\{ \begin{array}{rrr} - (U(x)-c)^{-1} \phi_{1,0}(z) \phi_{2,0}(x), 
\quad \mbox{if}\quad z>x,\\
- (U(x)-c)^{-1} \phi_{1,0}(x) \phi_{2,0}(z), \quad \mbox{if}\quad z<x.\end{array}\right.
$$ 
The inverse of $\Ray_0$ is explicitly given by
\begin{equation}\label{def-RayS0}
\begin{aligned}
RaySolver_0(f) (z)  &: =  \int_0^{+\infty} G_{R,0}(x,z) f(x) dx.
\end{aligned}
\end{equation}
Note that the Green kernel $G_{R,0}$ is singular at $z_c$.
The following lemma asserts that the operator $RaySolver_0(\cdot)$ 
is in fact well-defined from $X^{\eta}$ to $Y^{0}$, 
which in particular shows that $RaySolver_0(\cdot)$ gains two derivatives, but losses the fast decay at infinity.  
It transforms a bounded function into a function which behaves like $(z-z_c) \log(z-z_c)$ near $z_c$.

\begin{lemma}\label{lem-RayS0} 
Assume that $\Im c \not =0$. For any $f\in {X^{\eta}}$,   $RaySolver_0(f)$ 
is a solution to the Rayleigh problem \eqref{Ray0}. In addition, $RaySolver_0(f) \in Y^{0}$, and there holds  
$$
\| RaySolver_0(f)\|_{Y^{0}} \le C (1+|\log \Im c|) \|f\|_{{X^{\eta}}},
$$ 
for some constant $C$. 
\end{lemma}
\begin{proof} 
Using \eqref{decay-phi012}, it is clear that $\phi_{1,0}(z)$ and $\phi_{2,0}(z)/(1+z)$
  are uniformly bounded. Thus, considering the cases $x < 1$ and $x > 1$, we obtain 
\begin{equation}\label{est-Gr0}
|G_{R,0}(x,z)| \le  C \max\{ (1+x), |x-z_c|^{-1} \}.
\end{equation}
That is, $G_{R,0}(x,z)$ grows linearly in $x$ for large $x$ and has a singularity of order $|x-z_c|^{-1}$ when $x$ is near $z_c$.
As $|f(z)|\le e^{-\eta z} \| f\|_{X^{\eta}}$, the integral \eqref{def-RayS0} is well-defined and we have 
$$
|RaySolver_0(f) (z)| \le  C\| f\|_{X^{\eta}} \int_0^\infty e^{-\eta x} \max\{ (1+x), |x-z_c|^{-1} \}  \; dx
$$
$$
\le C (1+|\log \Im c|) \| f\|_{X^{\eta}},
$$
in which we used the fact that $\Im z_c \approx \Im c$. 
  
To bound the derivatives, we need to check the order of the singularity for $z$ near $z_c$. 
We note that 
$$
|\partial_z \phi_{2,0}| \le C (1+|\log(z-z_c)|),
$$ 
and hence 
$$
|\partial_zG_{R,0}(x,z)| \le  C \max\{ (1+x), |x-z_c|^{-1} \} (1+|\log(z-z_c)|).
$$
Thus, $\partial_z RaySolver_0(f)(z)$ behaves as $1+|\log(z-z_c)|$ near the critical layer. 
In addition, from the $\Ray_0$ equation, we have 
\begin{equation}\label{identity-R0f} 
\partial_z^2 (RaySolver_0(f)) = \frac{U''}{U-c} RaySolver_0(f) + \frac{f}{U-c}.
\end{equation}
This proves that $RaySolver_0(f) \in Y^{0}$ and gives the desired bound. 
\end{proof}


\subsection{Approximate Green function when $\alpha \ll1$}


Let $\phi_{1,0}$ and $\phi_{2,0}$ be the two solutions of $\Ray_0(\phi) = 0$ that are constructed above, in Lemma \ref{lem-defphi012}. 
We now construct an approximate Green function to the Rayleigh equation for $\alpha > 0$.
To proceed, let us introduce
\begin{equation}\label{def-phia12}
\phi_{1,\alpha } = \phi_{1,0} e^{-\alpha z} ,\qquad \phi_{2,\alpha} = \phi_{2,0} e^{-\alpha z}.
\end{equation}
A direct computation shows that their Wronskian determinant equals
$$
W[\phi_{1,\alpha},\phi_{2,\alpha}] =  \partial_z \phi_{2,\alpha} \phi_{1,\alpha} - \phi_{2,\alpha} \partial_z \phi_{1,\alpha}  = e^{-2\alpha z}.
$$ 
Note that the Wronskian vanishes at infinity since both functions have the same behavior at infinity. 
In addition, 
\begin{equation}\label{Ray-phia12}
\Ray_\alpha(\phi_{j,\alpha}) = - 2 \alpha (U-c) \partial_z \phi_{j,0} e^{-\alpha z} 
\end{equation}
We are then led to introduce an approximate Green function $G_{R,\alpha}(x,z)$, defined by 
$$
G_{R,\alpha}(x,z) = \left\{ \begin{array}{rrr} (U(x)-c)^{-1} e^{-\alpha (z-x)}  \phi_{1,0}(z) \phi_{2,0}(x), \quad \mbox{if}\quad z>x\\
(U(x)-c)^{-1} e^{-\alpha (z-x)}  \phi_{1,0}(x) \phi_{2,0}(z), \quad \mbox{if}\quad z< x.\end{array}\right.
$$
Again, like $G_{R,0}(x,z)$, the Green function $G_{R,\alpha}(x,z)$ is ``singular'' near $z_c$.
By a view of \eqref{Ray-phia12}, 
\begin{equation}\label{id-Gxz}
\Ray_\alpha (G_{R,\alpha}(x,z)) = \delta_{x}  + E_{R,\alpha}(x,z),
\end{equation}
for each fixed $x$, where the error kernel $E_{R,\alpha}(x,z)$ is defined by  
$$
E_{R,\alpha}(x,z) = 
 \left\{ \begin{array}{rrr}
- 2 \alpha (U(z) - c) (U(x)-c)^{-1} e^{-\alpha (z-x)} \partial_z \phi_{1,0}(z) \phi_{2,0}(x), \quad \mbox{if}\quad z>x\\
 - 2 \alpha (U(z) - c) (U(x)-c)^{-1}e^{-\alpha (z-x)}\  \phi_{1,0}(x) \partial_z \phi_{2,0}(z), \quad \mbox{if}\quad z< x.\end{array}\right.
$$
We then introduce an approximate inverse of the operator $\Ray_\alpha$ defined by
\begin{equation}\label{def-RaySa}
RaySolver_\alpha(f)(z) 
:= \int_0^{+\infty} G_{R,\alpha}(x,z) f(x) dx
\end{equation}
and the related error operator 
\begin{equation}\label{def-ErrR}
Err_{R,\alpha}(f)(z) := 2\alpha (U(z) - c) \int_0^{+\infty} E_{R,\alpha}(x,z) f(x) dx
\end{equation}

\begin{lemma}\label{lem-RaySa} 
Assume that $\Im c > 0$. For any $f\in {X^{\eta}}$,  with $\alpha<\eta$, 
the function $RaySolver_\alpha(f)$ is well-defined in $Y^{\alpha}$, and satisfies 
$$ 
\Ray_\alpha(RaySolver_\alpha(f)) = f + Err_{R,\alpha}(f).
$$
Furthermore, there hold  
\begin{equation}\label{est-RaySa}
\| RaySolver_\alpha(f)\|_{Y^{\alpha}} \le C (1+|\log \Im c|) \|f\|_{{X^{\eta}}},
\end{equation}
and 
\begin{equation}\label{est-ErrRa} 
\|Err_{R,\alpha}(f)\|_{Y^{\eta}} \le C | \alpha |   (1+|\log (\Im c)|)  \|f\|_{X^{\eta}} ,
\end{equation} 
for some universal constant $C$. 
\end{lemma}
\begin{proof}  
The proof follows that of Lemma \ref{lem-RayS0}.
 Indeed, since 
 $$
 G_{R,\alpha}(x,z)  = e^{-\alpha (z-x)} G_{R,0}(x,z),
 $$ 
 the behavior near the critical layer $z=z_c$ is the same for these two Green functions, 
 and hence the proof of \eqref{est-RaySa} and \eqref{est-ErrRa} near the critical layer identically follows  
 from that of Lemma \ref{lem-RayS0}.  

Let us check  the behavior at infinity. 
Consider the case $p=0$ and assume $\|f \|_{X^{\eta}} =1$. 
Using \eqref{est-Gr0}, we get
$$
|G_{R,\alpha}(x,z)| \le  C e^{-\alpha (z - x)} \max\{ (1+x), |x-z_c|^{-1} \}.
$$
Hence, by definition, 
$$
 |RaySolver_\alpha (f)(z) |\le C e^{-\alpha z} \int_0^\infty e^{\alpha x} e^{-\eta x}\max\{ (1+x), |x-z_c|^{-1} \}\; dx 
 $$ 
which is  bounded by $C(1+|\log \Im c|) e^{-\alpha z}$, upon recalling that $\alpha<\eta$.  
This proves the right exponential decay of $RaySolver_\alpha (f)(z)$ at infinity, for all $f \in X^{\eta}$. 

The estimates on $Err_{R,\alpha}$ are the same, once we notice that 
$(U(z) - c) \partial_z \phi_{2,0}$ has the same bound as that for $\phi_{2,0}$, and similarly for $\phi_{1,0}$. 
\end{proof}

\begin{remark}\label{rem-Ray} For $f(z) = (U-c) g(z)$ with $g\in {X^{\eta}}$, the same proof as done for Lemma  \ref{lem-RaySa} yields 
\begin{equation}\label{rem-RaySa}
\begin{aligned}
\| RaySolver_\alpha(f)\|_{Y^{\alpha}} &\le C \|g\|_{{X^{\eta}}},
\\
\|Err_{R,\alpha}(f)\|_{Y^{\eta}} &\le C | \alpha |   \|g\|_{X^{\eta}} 
\end{aligned}\end{equation} 
which are slightly better estimates as compared to \eqref{est-RaySa} and \eqref{est-ErrRa}. 
\end{remark}

\subsection{Construction of $\phi^{app}_{s,-}$ }\label{sec-exactRayleigh}


Let us start with the decaying solution $\phi_{s,-}$. We note that 
$$ 
\psi_{0} =  e^{-\alpha z} (U-c)
$$
is only a $O(\alpha)$ smooth approximate solution to Rayleigh equation since
$$
e_{0} =  Ray_\alpha (\psi_0) = - 2\alpha (U-c) U' e^{-\alpha z}.
$$
Similarly, a direct computation shows that 
$$
\OS_{\alpha,c}(\psi_0) = O(\alpha) = O(\nu^{1/4}).
$$ 
This is not sufficient for our purposes, and we have to go to the next
order. We therefore introduce
$$
\psi_1 = - RaySolver_\alpha(e_0).
$$
Note that $\psi_1$ is of order $O(\alpha)$ in $Y^\eta$, and behaves like $\alpha (z - z_c) \log (z - z_c)$ near $z_c$.
It particular $\psi_1$ is not a smooth function near $z_c$. Its fourth order derivative behaves like
$\alpha / (z - z_c)^3$ in the critical layer. We have 
$$
\OS_{\alpha,c}(\psi_1) = \eps (\partial_z^2 - \alpha^2)^2 \psi_1 + Ray_\alpha(\psi_1) .
$$
hence 
\beq \label{errr1}
\OS_{\alpha,c}(\psi_0 + \psi_1) = \eps (\partial_z^2 - \alpha^2)^2 \psi_1 + Err_{R,\alpha}(e_0) .
\eeq
Note that
\beq \label{Raypsi1}
 Err_{R,\alpha}(e_0) =  O ( \alpha^2 | \log(\alpha) |)_{Y^\eta} .
\eeq
Moreover, using Rayleigh equation,
$$
(\partial_z^2 - \alpha^2) \psi_1 = {Ray_\alpha(\psi_1) - U'' \psi_1 \over U -c },
$$
hence
\beq \label{errr}
E := \eps (\partial_z^2 - \alpha^2)^2 \psi_1 
= \eps (\partial_z^2 - \alpha^2)   \Bigl\{ {Ray(\psi_1) - U'' \psi_1 \over U -c } \Bigr\} .
\eeq
In view of Remark \ref{rem-Ray}, $Ray_\alpha(\psi_1)$ and $U'' \psi_1$ are of order $O(\alpha)$ in $X^\eta$. We thus have
$$
\eps \alpha^2 \Bigl| {Ray(\psi_1) - U'' \psi_1 \over U -c } \Bigr|
\le C {\eps \alpha^2 \over | z - z_c | } 
\le C {\eps \alpha^2 \over | \Im c | } \le C \eps \alpha = O(\nu)_{X^\eta}.
$$
Next we expand $\partial_z^2$ in (\ref{errr}) which gives three terms. The first one is
$$
\eps {\partial_z^2 Ray(\psi_1) - \partial_z^2(U'' \psi_1) \over U - c} .
$$
As $Ray_\alpha(\psi_1)$ and $\psi_1$ are of order $O(\alpha)$ in $Y^\eta$, this quantity is bounded by
\beq \label{bound1}
C \eps \Bigl( 1 + {\alpha | \log \Im c | \over |z - z_c|}+ {\alpha  \over |z - z_c|^2} \Bigr) \le C {\eps \alpha  \over | \Im c|^2} = O(\alpha^2 ). 
\eeq
The third term in the expansion of (\ref{errr}) is
$$
\eps \Bigl[  Ray(\psi_1) - U'' \psi_1 \Bigr] (z - z_c)^{-3}
$$
which is bounded by $O(\alpha)$. 
Thus, we can write the error term as 
$$ E = E_1 + E_2, \qquad E_1 = O(\alpha^2), \qquad E_2\le C \eps \alpha | z - z_c |^{-3} .
$$
This error term $E_2$ is therefore too large for our purposes. However, it is located near $z = z_c$, namely in the critical
layer.  We therefore correct $\psi_0 + \psi_1$ by $\psi_2$ by approximately inverting the Airy operator in this layer. 
More precisely, let
$$
\psi_2 = - AirySolve(E_2) ,
$$
which will create an error term
$$
\begin{aligned}
E_3 &=  \OS_{\alpha,c} (\psi_2) + E_2 
\\& = \Airy (\psi_2) + \hbox{OrrAiry} (\psi_2) + E_2
\\&=  \hbox{OrrAiry} (\psi_2) + ErrorAiry(E_2).
\end{aligned}$$
Let us now bound $\psi_2$. Using \eqref{AirySolve}, we have
$$
| \psi_2 (y) |\le C \eps \alpha \int_0^{+ \infty} | x - z_c |^{-3}
 \Bigl(e^{ - | \int_x^y \Re \mu_f(z)dz |}   + O(\nu^{-1/4}) 1_{y < x} x \Bigr) dx .
$$
Writing $1_{y < x} x = 1_{y < x} (x - z_c) + 1_{y < x} z_c$, we thus have 
$$
\begin{aligned}
| \psi_2 (y) | &\le C \eps \alpha \int_0^{+ \infty} \Bigl ( | x - z_c |^{-3} + \nu^{-1/4} |x-z_c|^{-2}\Bigr )
 dx 
 \\
 &\le C \eps \alpha \Bigl( |\Im c|^{-2}+ \nu^{-1/4} |\Im c|^{-1}\Bigr )  = O( \alpha^2)
 .\end{aligned}
$$
This together with \eqref{Airy-de-2} yields $ \hbox{OrrAiry} (\psi_2) = O( \alpha^2)$.  Similarly, using (\ref{GreenAiry2}), we get 
$$
ErrorAiry(E_2)(z) \le C \eps \alpha\int_0^{+ \infty} | x - z_c |^{-3} O(\nu^{1/4}) dx = O(\alpha^3).
$$
Therefore, we have
$$
\OS_{\alpha,c}(\psi_0 + \psi_1 + \psi_2) = O(\alpha^2) .
$$
We define
$$
\phi_{s,-}^{app} = \psi_0 + \psi_1 + \psi_2 .
$$
To end this section we compute $\psi(0)$. By definition,
$$
\begin{aligned}
\psi_1(0) &= - RaySolver_\alpha(e_0) (0)
 = - \phi_{2,\alpha}(0) \int_0^{+\infty} e^{2\alpha x}\phi_{1,\alpha}(x) {e_0(x) \over U(x) - c} dx 
\\&=  -  2 \alpha \phi_{2,0}(0) \int_0^{+\infty}  U' (U-c)dz =
 \alpha  \phi_{2,0}(0) \Bigl[ (U - c)^2 \Bigr]_0^{+ \infty}
\\&= - \alpha \phi_{2,0}(0) \Bigl[ (U_+ - c)^2 - c^2 \Bigr]  
=   \alpha {U_+  \over U'(0)} (U_+ - 2 c) .
\end{aligned}$$
From the definition, we have 
$$
\phi_{s,-}(0) = U_0 -c + \psi_1(0) + O(\alpha^2).
$$
This proves the lemma,  using that $U_0 - c = O(z_c)$.


\subsection{Construction of $\phi^{app}_{s,+}$ \label{sec35}}


We first start with $\phi_{2,\alpha} = \phi_{2,0} e^{- \alpha z}$, which is an approximate solution of Rayleigh equation,
up to a $O(\alpha)$ error term.
Let 
$$
e_1 = Ray_\alpha (\phi_{2,\alpha}) = O(\alpha).
$$
We introduce
$$
\phi_3 = - RaySolver_\alpha(e_1) .
$$
Then 
$$
Ray_\alpha( \phi_{2,\alpha} + \phi_3) = -Err_{R,\alpha}(e_1) = O(\alpha^2).
$$
Let
$$
\phi_{s,+} = \phi_{2,\alpha} + \phi_3.
$$
Note that $\phi_{s,+}$ is bounded in $Y^\eta$, and thus behaves like $(z - z_c) \log (z - z_c)$ near $z_c$


We have
$$
\OS_{\alpha,c}(\phi_{s,+}) = - \eps (\partial_z^2 - \alpha^2)^2 \phi_{s,+},
$$
but, using Rayleigh equation,
$$
(\partial_z^2- \alpha^2) \phi_{s,+} = {U'' \over U - c} \phi_{s,+},
$$
hence
$$
(\partial_z^2- \alpha^2)^2 \phi_{s,+} = (\partial_z^2- \alpha^2) \Bigl( {U'' \over U - c} \phi_{s,+} \Bigr).
$$
The worst term in the right hand side is
$$
\Bigl[ \partial_z^2 \Bigl( { U'' \over U - c} \Bigr) \Bigr] \phi_{s,+} 
$$
which is of order $(\Im z_c)^{-3}$ for $\phi_{s,-}$. Hence $\OS_{\alpha,c}(\phi_{s,+})$ is of order
$$
{\eps \over (\Im z_c)^3}
\sim {\nu \over \alpha} {1 \over \nu^{3/4}} \sim 1 
$$
near $z = z_c$, which is $\alpha^{-1}$ larger than for $\phi_{s,-}^{app}$. 
As a consequence we loose a factor $\alpha^{-1}$ in the end of the construction with respect to $\phi_{s,-}^{app}$.
The construction is similar, up to a factor $\alpha^{-1}$.


\section{Green function for Orr-Sommerfeld equations}


We now construct an approximate Green function $G^{app}$ using the approximate solutions
$\phi_{s,\pm}^{app}$ and $\phi_{f,\pm}^{app}$. We will decompose this Green function into two components
$$
G^{app} = G_i^{app} + G_b^{app}
$$
where $G_i^{app}$ takes care of the source term $\delta_x$ and where $G_b^{app}$ takes care of the boundary conditions.


\subsection{Interior approximate Green function}


We look for $G_i^{app}(x,y)$ of the form
$$
G_i^{app}(x,y) = a_+(x)  \phi_{s,+}^{app}(y) 
+ {b_+(x) }  {\phi_{f,+}^{app}(y) \over \phi_{f,+}^{app}(x)} 
\quad \hbox{for} \quad y < x,
$$
$$
G_i^{app}(x,y) = a_-(x)  \phi_{s,-}^{app}(y)
+ {b_-(x)} {\phi_{f,-}^{app}(y) \over \phi_{f,-}^{app}(x)} 
\quad \hbox{for} \quad y  > x,
$$
where $\phi_{f,\pm}^{app}(x)$ play the role of normalization constants. Let
$$
F_\pm = \phi^{app}_{f,\pm}(x) 
$$
and let 
$$
v(x) = (- a_-(x), a_+(x), - b_-(x), b_+(x) ) .
$$
By definition of a Green function, $G^{app}$, $\partial_y G^{app}$ and $\partial_y^2 G^{app}$ are continuous at $x = y$,
whereas $- \eps \partial_y^3 G^{app}$ has a unit jump at $x = y$. 
Let
\beq \label{matriceM}
M = \left( \begin{array}{cccc} 
\phi_{s,-} & \phi_{s,+} & \phi_{f,-} / F_-  & \phi_{f,+} /F_+  \cr
\partial_y \phi_{s,-}  / \mu_f & \partial_y\phi_{s,+} /   \mu_f
& \partial_y\phi_{f,-} /  F_- \mu_f & \partial_y\phi_{f,+} / F_+ \mu_f  \cr
\partial_y^2 \phi_{s,-} / \mu_f^2 &\partial_y^2 \phi_{s,+} /   \mu_f^2 
& \partial_y^2 \phi_{f,-} /  F_- \mu_f^2  & \partial_y^2 \phi_{f,+} /  F_+ \mu_f^2 \cr
 \partial_y^3 \phi_{s,-} /  \mu_f^3 &  \partial_y^3 \phi_{s,+} /   \mu_f^3
& \partial_y^3 \phi_{f,-}  /  F_- \mu_f^3 &  \partial_y^3 \phi_{f,+} /  F_+ \mu_f^3 \cr 
\end{array} \right) ,
\eeq
where the functions $\phi_{s,\pm} = \phi_{s,\pm}^{app}$ and $\phi_{f,\pm} = \phi_{f,\pm}^{app}$ 
and their derivatives are evaluated at $y=x$, and where the various factors $\mu_f$ are introduced to renormalize the lines of $M$.
Then 
\beq \label{Mv}
M v = (0,0,0,- 1/ \eps \mu_f^3) .
\eeq
We will  evaluate $M^{-1}$ using the following block structure. 
Let $A$, $B$, $C$ and $D$ be the two by two matrices defined by
$$
M = \left( \begin{array}{cc} 
A & B \cr
C & D \cr \end{array} \right) .
$$
We will prove that $C$ is small, that $D$ is invertible and that $A$ is related to Rayleigh equations.
This will allow the construction of an explicit approximate inverse, and by iteration, of the inverse of $M$. 
Let us detail these points.

Let us first study $D$. Following  (\ref{fastapp1c}), for $z \gg \nu^{1/4}$,
$$
D = \left( \begin{array}{cc}
1 & 1 \cr
-1 & 1 \cr 
\end{array} \right) + o(1),
$$
hence $D$ is invertible and
$$
D^{-1} =  \left( \begin{array}{cc}
1 & -1 \cr
1 & 1 \cr 
\end{array} \right) + o(1).
$$
For $z$ of order $\nu^{1/4}$, we note that $F_+$ and $F_-$ are of order $O(1)$,
$$
\partial_y^2 \phi_{f,-} = \gamma^2 {1 \over (g')^2(x)} {Ai ( \gamma g(x) ) \over Ai(2, \gamma g(x)) } + O(\gamma)
$$
and similarly for $\partial_y \phi_{f,-}$ and $\partial_y \phi_{f,+}$. Note that $\gamma^2 / \mu_f^2$, $\gamma^3 / \mu_f^3$,
$Ai(2,\gamma g(x))$ and $Ci(2,\gamma g(x))$ are of order $O(1)$. As $g'(z_c) = 1$, up to normalization of
lines and columns, $D$ is close to
$$
\left( \begin{array}{cc}
Ai & Ci \cr
Ai' & Ci' \cr 
\end{array} \right) 
$$
which is invertible by definition of the special Airy functions $Ai$ and $Ci$.

Let us turn to $C$. The worst term in $C$ is those involving $\phi_{s,+}$ because of its logarithmic singularity.
More precisely, $\partial_y^k \phi_{s,+}$ behaves like $(z-z_c)^{k-1}$ and
is bounded by $| \Im c |^{k-1} \sim \nu^{(1 - k) / 4}$ for $k = 2$, $3$.
Hence, as $\mu_f^{-1} = O(\nu^{1/4})$, 
$$
C =  \left( \begin{array}{cc}
O(\nu^{1/2 }) & O(\nu^{1/2 }  (z-z_c)^{-1})  \cr
O(\nu^{3/4 }) & O(\nu^{3/4 } (z - z_c)^{-2}) \cr
\end{array} \right) 
$$
Note that $A = A_1 A_2$ with
$$
A_1 = \left( \begin{array}{cc}  
1 & 0 \cr
0 & \mu_f^{-1} \cr
\end{array} \right),
\quad 
A_2 = \left( \begin{array}{cc}  
\phi_{s,-}^{app} & \phi_{s,+}^{app} \cr
\partial_y \phi_{s,-}^{app}  & \partial_y \phi_{s,+}^{app}  \cr
\end{array} \right) .
$$
We have
$$
A_2^{-1} = {1 \over \det(A_2)} \left( \begin{array}{cc}  
\partial_y \phi_{s,+}^{app}  & - \phi_{s,+}^{app} \cr
- \partial_y \phi_{s,-}^{app}  & \phi_{s,-}^{app}  \cr
\end{array} \right) .
$$
The determinant  $A_2$ is the Wronskian of $\phi_{s,\pm}^{app}$ and hence a perturbation
of the Wronskian of $\phi_{1,\alpha}$ and $\phi_{2,\alpha}$ which equals to $e^{- \alpha x}$.
We distinguish between $x < \alpha^{1/2}$ and $x > \alpha^{1/2}$. In the second case, $\OS_{c,\alpha}$
is a small perturbation of a constant coefficient fourth order operator. The Green function may therefore be explicitly computed.
We will not detail the computations here and focus on the case where $x < \alpha^{1/2}$. In this case the Wronskian is of order $O(1)$.
As a consequence 
$$
A_2^{-1} = \left( \begin{array}{cc}  
O(\log | z - z_c | )  & O(1) \cr
O(1)  & O(z - z_c)  \cr
\end{array} \right) 
$$
and
$$
A^{-1} = \left( \begin{array}{cc}  
O(\log | z - z_c |)  & O(\mu_f) \cr
O(1)  & O(\mu_f (z - z_c))  \cr
\end{array} \right) 
$$
We now observe that the matrix $M$ has an approximate inverse
$$
\widetilde M = \left( \begin{array}{cc}
A^{-1} & - A^{-1} B D^{-1}  \cr
0 & D^{-1} \cr 
\end{array} \right) 
$$
in the sense that $M  \widetilde M = Id + N$ where
$$
N =  \left( \begin{array}{cc}
0 & 0 \cr 
 C A^{-1} &  - C A^{-1} B D^{-1} \cr 
\end{array} \right) .
$$
Now a direct calculation shows that 
$$
C A^{-1} = O(\nu^{1/4})
$$ 
since $\Im z_c = O(\nu^{1/4})$. 
As $D^{-1}$ and $B$ are uniformly bounded,
$N = O(\nu^{1/4})$. In particular, $(Id + N)^{-1}$
is well defined and 
$$
M^{-1} = \widetilde M (Id + N)^{-1} = \widetilde M \sum_n N^n.
$$
Note that the two first lines of $N^n$ vanish. The other lines are at most of order $O(\nu^{1/4})$.  Therefore
$$
(Id + N)^{-1} (0,0,0, 1 / \nu \mu_f^3) = \Bigl( 0, 0, O(1 / \nu \mu_f^4), 1 / \nu \mu_f^3 \Bigr) .
$$
As $D^{-1}$ is bounded and $A^{-1} B D^{-1}$ is of order $O(\mu_f)$, we obtain that 
$a_\pm$ and $b_\pm$ are respectively bounded by $C / \nu \mu_f^2$ and $C / \nu \mu_f^3$.


\subsection{Boundary approximate Green function}


We now add to $G_i^{app}$ another Green function $G_b^{app}$ to handle the boundary conditions.
We look for $G_b^{app}$ under the form
$$
G_b^{app}(y) = d_s \phi_{s,-}(y)  + d_f {\phi_{f,-}(y)  \over \phi_{f,-}(0)},
$$
where $\phi_{f,-}(0)$ in the denominator is a normalization constant, 
and look for $d_s$ and $d_f$ such that
\beq \label{Greenb1}
G_i^{app}(x,0) + G_b^{app}(0) = 
\partial_y G_i^{app}(x,0) + \partial_y G_b^{app}(0) =  0.
\eeq
Let
$$
M = \left( \begin{array}{cc} \phi_{s,-} & \phi_{f,-} / \phi_{f,-}(0) \cr
\partial_y \phi_{s,-} & \partial_y \phi_{f,-} / \phi_{f,-}(0) \cr \end{array} \right) ,
$$
the functions being evaluated at $y = 0$. Then (\ref{Greenb1}) can be rewritten as
$$
M d = - (G_i^{app}(x,0), \partial_y G_i^{app}(x,0)) 
$$
where $d = (d_s,d_f)$. Note that
$$
(G_i^{app}(x,0), \partial_y G_i^{app}(x,0))  = Q (a_+,b_+)
$$
where
$$
Q =  \left( \begin{array}{cc}
 \phi_{s,+}(0)  & 1 \cr
 \partial_y \phi_{s,+}(0)   & \partial_y \phi_{f,+}(0) / \phi_{f,+}(0) \cr
\end{array} \right) 
= \left( \begin{array}{cc}
 O(1)  & 1 \cr
 O( \log( \nu))   & O(\nu^{-1/4}) \cr
\end{array} \right) .
$$
By construction
\beq \label{defid}
d = -   M^{-1} Q (a_+,b_+) .
\eeq
We have
$$
M^{-1} = {1 \over \det(M)}  \left( \begin{array}{cc} \partial_y \phi_{f,-}(0) / \phi_{f,-}(0) & - 1 \cr
- \partial_y \phi_{s,-}(0)& \phi_{s,-}(0)  \end{array} \right) .
$$
The determinant of $M$ equals
$$
\det M = {W[\phi_{s,-},\phi_{f,-}](0) \over \phi_{f,-}(0)} 
$$
and does not vanish by assumption. Therefore
$$
M^{-1} = \left( \begin{array}{cc} O(\nu^{-1/4}) & - 1 \cr
O(1) & O(\nu^{1/4}) \end{array} \right) .
$$
As a consequence,
$$
M^{-1} Q =
\left( \begin{array}{cc} O(\nu^{-1/4}) & O(\nu^{-1/4}) \cr
O(1) & O(1) \end{array} \right).
$$


\subsection{Exact Green function}


Once we have an approximate Green function, we obtain the exact Green function by iteration, following the strategy developped
in \cite{GrN1}.




\end{document}